\documentclass[11pt]{article}

\usepackage{amsmath,amssymb,amsthm,bm}
\usepackage{natbib}
\usepackage{enumitem} 
\usepackage{breakcites}
\usepackage[colorlinks,
linkcolor=blue,
anchorcolor=blue,
citecolor=blue,
breaklinks=true
]{hyperref}

\usepackage{geometry}
\def\T{ \mathrm{\scriptscriptstyle T} }
\def\E{\mathbb{E}} 
\def\P{\mathbb{P}} 

\newtheorem{theorem}{Theorem}
\newtheorem{proposition}{Proposition}
\newtheorem{lemma}{Lemma}
\newtheorem{assumption}{Assumption}

\begin{document}

\title{ \LARGE Adaptive Huber Regression on Markov-dependent Data}

\author{Jianqing Fan,~Yongyi Guo,~and~Bai Jiang\thanks{Department of Operations Research and Financial Engineering, Princeton University, 98 Charlton Street, Princeton, NJ 08540; E-mail: \texttt{\{jqfan,yongyig,baij\}@princeton.edu}. The research is supported by DMS-1662139 and DMS-1712591 and NIH grant 2R01-GM072611-14.}}


\date{ }

\maketitle

\vspace{-0.25in}

\begin{abstract}
High-dimensional linear regression has been intensively studied in the community of statistics in the last two decades. For the convenience of theoretical analyses, classical methods usually assume independent observations and sub-Gaussian-tailed errors. However, neither of them hold in many real high-dimensional time-series data. Recently [Sun, Zhou, Fan, 2019, J. Amer. Stat. Assoc., in press] proposed Adaptive Huber Regression (AHR) to address the issue of heavy-tailed errors. They discover that the robustification parameter of the Huber loss should adapt to the sample size, the dimensionality, and the moments of the heavy-tailed errors. We progress in a vertical direction and justify AHR on dependent observations. Specifically, we consider an important dependence structure -- Markov dependence. Our results show that the Markov dependence impacts on the adaption of the robustification parameter and the estimation of regression coefficients in the way that the sample size should be discounted by a factor depending on the spectral gap of the underlying Markov chain.
\end{abstract}

\section{Introduction}
In the Big Data era, massive and high-dimensional data characterize many modern statistical problems, arising from biomedical sciences, econometrics, finance, engineering and social sciences. Examples include the gene expression data, the functional magnetic resonance imaging (fMRI) data, the macroeconomic data, the high-frequency financial data, the high-resolution image data, the e-commerce data, among others. The data abundance, the high dimensionality and other complex structures have given rise to a few statistical and computational challenges \citep*{fan2014challenges}.

An important and fundamental problem is high-dimensional linear regression, in which the dimensionality (the number of covariates) is much larger than the sample size so that the ordinary least squares estimation of the regression coefficients is highly unstable. To solve this problem, statisticians have made extensive progresses in the development of high-dimensional statistical inference \citep*{tibshirani1996regression, fan2001variable, candes2007dantzig, bickel2009simultaneous, su2016slope}. They commonly make the sparsity assumption that only a small number of covariates contribute to the response variable, and enforce the sparsity of regression coefficients by regularization techniques. For a comprehensive and systematic overview of this high-dimensional regression problem, we refer to \citet{buhlmann2011statistics} and \citet*{tibshirani2015statistical}.

For the convenience of theoretical analyses, many high-dimensional regression methods assume Gaussian or sub-Gaussian-tailed errors in the model. However, this assumption is unsatisfied in a broad range of real datasets. See \citet{cont2001empirical} for asset return data, \citet*{stock2002macroeconomic, mccracken2016fred} for macroeconomic data, \citet{gupta2014transcriptome} for RNA-seq gene expression data, \citet*{wang2015high} for microarray gene expression data, \citet*{eklund2016cluster} for fMRI data, to name a few. \cite*{fan2016shrinkage} argue that heavy-tailed errors are stylized features of high-dimensional data. These heavy-tailed errors impair the consistency of many high-dimensional regression methods, as the ordinary least squares loss function is non-robust to outliers.

Robust loss functions, which are less sensitive to outliers, have been considered to address the issue of heavy-tailed errors. Huber's seminar work \citep{Huber1964robust, Huber1973robust} introduced robust M-estimators (``M'' for ``maximum likelihood-type'') and provided the initial theory on robust regression methods. Asymptotic properties of robust M-estimators in the low-dimensional setting have been well studied by \citet{yohai1979asymptotic, portnoy1985asymptotic, mammen1989asymptotics, he1996general, bai1997general, he2000parameters}. Recently \citet{loh2017statistical} investigated theoretical properties of a general class of regularized robust M-estimators for the high-dimensional linear model with heavy-tailed errors.

More recently, \citet*{sun2019adaptive} studied a specific regularized robust M-estimator, which minimizes the $\ell_1$-regularized Huber loss, for the high-dimensional linear model with heavy-tailed errors, which admit finite $(1+\delta)$-moments for some $\delta > 0$. Remarkably, they observe that the robustification parameter of the Huber loss should adapt to the sample size, the dimensionality and $(1+\delta)$-moments of heavy-tailed errors for an optimal tradeoff between bias and robustness. The adaption of the robustification parameter exhibits a smooth phase transition between regimes of $0 < \delta < 1$ and $\delta > 1$. Thereafter, they name their method \textit{Adaptive Huber Regression} (AHR), to highlight its difference from others' Huber regression methods with fixed robustification parameter.

Apart from heavily-tailed errors, another common feature of high-dimensional data, especially those collected in a temporal order, is dependent observations. For example, fMRI time series data are usually collected from a few different regions over a time period \citep{friston1995analysis, worsley1995analysis}; the monthly data of macroeconomic variables spanning the time period of decades are now benchmark datasets used by many econometric studies \citep*{stock2002macroeconomic, ludvigson2009macro, mccracken2016fred}.

High-dimensional regression has been applied to fMRI time-series data \citep{smith2012future, ryali2012estimation, tang2012measuring} and economic time-series data \citep{ludvigson2009macro,fan2011sparse,belloni2012sparse}, although its theoretical framework takes little consideration of the dependence structure of observations. For fMRI time-series data, the lasso or elastic-net regression method is one of the main tools to find a small number of functionally-connected regions of a specific region in human brains from temporally-ordered observations across all the regions \citep{smith2012future}. On a macroeconomic dataset, \citep{ludvigson2009macro} regressed the U.S. bond premia on macroeconomic variables over the time period spanning from January, 1964 to December, 2003. Nevertheless, the applicability of high-dimensional regression methods to these time-series data is not fully understood, because these methods assume either unconditional independence of data samples directly (as in the random design setup), or conditional independence of data samples given covariates satisfying some conditions (as in the fixed design setup) \citep*{buhlmann2011statistics, tibshirani2015statistical}. In the latter case, high probabilities of these conditions are usually verified with independent observations of covariates.

This paper aims to close the gap between theories of high-dimensional regression methods and practical needs in addressing both heavy-tailed errors and dependent observations of real data. Inspired by the optimality of AHR dealing with heavy-tailed errors in the independent setup, we consider extending AHR to cope with the dependence structure of observations. Albeit of the theoretical results on AHR by \citet*{sun2019adaptive}, it is still unclear whether AHR works for the dependent data. This unclearness puts the applicability of AHR to many real high-dimensional datasets in doubt. Even if AHR is justifiable under some type of dependence structure, we are still curious about the degree to which the data dependence influences the error rate of the AHR estimator. As an initial step towards full answers for these questions, we narrow down to the Markov dependence, an important and widely-used dependence structure, and analyze AHR on Markov-dependent data. Specially, we assume that covariates are functions of an underlying Markov chain and heteroskedastic heavy-tailed errors are dependent on the Markov chain.

In this Markov-dependent setup, we show under moderate conditions that AHR exhibits a similar phase transition of the adaptation of the robustification parameter and the estimation of regression coefficients between regimes of $0 < \delta < 1$ and $\delta > 1$, compared to that in the usual independent setup. The only difference is that the sample size should be discounted by a factor depending on the spectral gap of the underlying Markov chain.

The core of the proof is to bound in $\ell_\infty$-norm the gradient of the Huber loss at the true sparse vector of coefficients, denoted by $\bm{\beta}^\star$, and to establish the restricted eigenvalue condition of the Hessian of the Huber loss over a neighborhood of $\bm{\beta}^\star$. In the usual independent setting, these tasks are accomplished by applying Bernstein's inequality for independent random variables. The Bernstein-type mixture of the sub-Gaussian and sub-exponential tails is the key to derive the trade-off of bias and robustness in AHR. However, the analogous tasks in the Markov-dependent setup are non-trivial, due to the lack of Bernstein's inequalities for (possibly non-identical) functions of Markov chains. A very recent work \citep*{jiang2018bernstein} establishes the exact counterpart of Bernstein's inequality for \textbf{bounded}, non-identical functions of Markov chains. But it still does not fully meet the requirements of the theoretical analyses in this paper, because covariates and errors involved in the theoretical analyses are \textbf{unbounded}. We develop a truncation argument for the extension of the Berstein-type inequality in \citep*{jiang2018bernstein} to unbounded functions.

The rest of this paper is organized as follows. Section \ref{sec2} introduces the high-dimensional linear model and the methodology of AHR in the Markov-dependent setup. Section \ref{sec3} presents the assumptions and the main theorem. Section \ref{sec4} sketches the proof of the main theorem. Other technical proofs are collected in Section \ref{sec5}. Section \ref{sec6} concludes the paper with a brief discussion.

\section{Model and Methodology}\label{sec2}
The goal is to estimate the high-dimensional linear model with Markov-dependent covariates and heavy-tailed errors. Start with the linear model as follows.
$$y_i = \bm{x}_i^\T\bm{\beta} + \varepsilon_i,~~~i=1,\dots,n,$$
where $y_i$ is the response, $\bm{x}_i \in \mathbb{R}^d$ is the vector of $d$ covariates, $\varepsilon_i$ is the error, $\bm{\beta} \in \mathbb{R}^d$ is the vector of $d$ regression coefficients. In the high-dimensional regime, $d$ is much larger than the sample size $n$. To make the model identifiable, assume only a small number $s$ of covariates contribute to the response, i.e., the vector of true regression coefficients $\bm{\beta}^\star$ contains at most $s$ non-zero elements.

Suppose covariates $\{\bm{x}_i\}_{i=1}^n$ are functions of a stationary Markov chain $\{Z_i\}_{i=1}^n$ on a general state space, i.e., for a collection of $d$-dimensional vectorial functions $\{\bm{f}_i\}_{i=1}^n$, where $\bm{f}_i = (f_{i1},\dots,f_{id})^\T$,
$$x_{ij} = f_{ij}(Z_i). $$
The errors are conditionally independent given the underlying Markov chain and possibly heteroskedastic. There exists a conditional distribution $g(\cdot |z)$ such that
$$\varepsilon_i | Z_i \sim g(\cdot|Z_i).$$
If $\{Z_i\}_{i=1}^n$ are independently and identically distributed (i.i.d.) then this regression setup reduces to the usual one in which $\{(\bm{x}_i, \varepsilon_i)\}_{i=1}^n$ are independent.

For the task of estimating the true $s$-sparse coefficients $\bm{\beta}^\star$, consider the $\ell_1$-regularized robust M-estimator with the Huber loss as follows.
\begin{equation} \label{Huber1}
\widehat{\bm{\beta}}_{\tau, \lambda} = \arg \min_{\bm{\beta}} H_\tau(\bm{\beta}) + \lambda\Vert \bm{\beta}\Vert_1, ~~~ H_\tau(\bm{\beta}) = \frac{1}{n}\sum_{i=1}^n h_\tau(y_i - \bm{x}_i^\T\bm{\beta}),
\end{equation}
where $h_\tau$ is the Huber loss \citep{Huber1964robust}
\begin{equation} \label{Huber2}
h_\tau(w) =
\begin{cases}
w^2/2 & \mbox{if } |w| \le \tau\\
\tau |w| - \tau^2/2 & \mbox{if } |w| > \tau,
\end{cases}
\end{equation}
with the so-called robustification parameter $\tau$ \citep*{fan2017estimation}, and $\lambda$ is the regularization parameter encouraging the sparsity of $\bm{\beta}$ \citep{tibshirani1996regression}.

Heuristically, a larger robustification parameter $\tau$ reduces the bias of $\widehat{\bm{\beta}}_{\tau, \lambda}$ at the cost of less robustness. The extreme case of $\tau=\infty$ corresponds to the ordinary least squares estimation. We find that $\tau$ should adapt to the sample size $n$, the dimensionality $d$, the heaviness of the tails of the errors and the dependence of the Markov chain. Suppose the heavy-tailed errors admit finite $(1+\delta)$-moments (conditionally on $Z_i$'s) for some $\delta > 0$. A large $\delta$ indicates light tails of errors. The dependence of the Markov chain is measured by a quantity $\gamma \in [0,1]$, denoting the norm of the Markov operator (induced by transition kernel) acting on the Hilbert space of all squared-integrable and mean-zero functions with respect to the invariant distribution. A small $\gamma$ indicates a fast convergence of the Markov chain towards its stationarity from a non-stationary initial distribution \citep{rudolf2012explicit}. Our analyses show under moderate conditions that the choice of
$$\tau \asymp \left(\frac{1-\gamma}{1+\gamma} \cdot \frac{n}{\log d}\right)^{1/(1+\min\{\delta,1\})}.$$
achieves the optimal trade-off between bias and robustness in the Markov dependent setup.

\section{Assumptions and Theorems} \label{sec3}
This section presents our main result under four assumptions. The first assumption is on the convergence speed of the underlying Markov chain.

\begin{assumption}[Markov chain with non-zero spectral gap] \label{asm1}
The underlying Markov chain $\{Z_i\}_{i=1}^n$ is stationary with its unique invariant measure $\pi$ and admits a non-zero spectral gap $1-\gamma$.
\end{assumption}
Recall that the quantity $\gamma$ is defined as the norm of the Markov operator (induced by transition kernel) acting on the Hilbert space of all $\pi$-squared-integrable and $\pi$-mean-zero functions. $1-\gamma$ is called spectral gap of the Markov chain. It has been involved as constants in mean squared error bound for Markov chain Monte Carlo \citep{rudolf2012explicit}, Hoeffding-type and Bernstein-type inequalities for Markov chains \citep*{lezaud1998chernoff, leon2004optimal, chung2012chernoff, miasojedow2014hoeffding, paulin2015concentration, fan2018hoeffding, jiang2018bernstein}. A non-zero spectral gap is closely related to other convergence criteria of Markov chains \citep*{roberts1997geometric, roberts2001geometric, kontoyiannis2012geometric}.

Next two assumptions allow both covariates and heavy-tailed errors to be heteroskedastic, but impose on them some moment conditions. For the covariates, the bounded fourth moment of some envelop function is required. For the heavy-tailed errors, the bounded (conditional) $(1+\delta)$-moment is required.

\begin{assumption}[Covariates with bounded fourth moments] \label{asm2}
There exists an envelop function $M: z \mapsto \mathbb{R}$ for functions $f_{ij}$'s, i.e., $M(z) \ge \max_{1 \le i \le n, 1 \le j \le d} |f_{ij}(z)|$ for $\pi$-almost every $z$. And, $\sigma^4 := \int M^4(z) \pi(dz) < \infty$.
\end{assumption}

\begin{assumption}[Errors with bounded $(1+\delta)$-moments] \label{asm3}
$\E[\varepsilon_i |Z_i] = 0$ almost surely, i.e., $\int \varepsilon g(\varepsilon|z)d\varepsilon = 0$ for $\pi$-almost every $z$. And, for some $\delta >0$ and $v_\delta > 0$, $\E[|\varepsilon_i|^{1+\delta} |Z_i] < v_\delta$ almost surely,  i.e., $\int |\varepsilon|^{1+\delta} g(\varepsilon|z)d\varepsilon < v_\delta$ for $\pi$-almost every $z$.
\end{assumption}

The last assumption is on the restricted eigenvalue of the (aggregated) covariance matrix of covariates. It is a unified condition in the literature of high-dimensional regression, see e.g., \citet*{bickel2009simultaneous} and \cite{fan2018lamm}. Let $S = \{j: \beta_j^\star \ne 0\}$ be the index set of active covariates. Define the $\ell_1$-cone
\begin{equation} \label{cone}
\mathcal{C} := \{\bm{u} \in \mathbb{R}^d: \Vert \bm{u}_{S^c}\Vert_1 \le 3 \Vert \bm{u}_S \Vert_1 \},
\end{equation}
where $\bm{u}_{S^c}$ is the subvector assembling $\{\bm{u}_j: j \in S^c\}$ and $\bm{u}_S$ is the subvector assembling $\{\bm{u}_j: j \in S\}$. The constant $3$ in the definition of $\mathcal{C}$ has no specific meaning. It can be replaced by other constant larger than $1$. Write the (aggregated) covariance matrix of covariates as
\begin{equation} \label{cov1}
\bm{\Sigma}_n :=  \frac{1}{n} \sum_{i=1}^n \E[\bm{x}_i\bm{x}_i^\T].
\end{equation}
\begin{assumption}[Restricted eigenvalue of covariance matrix] \label{asm4}
There exists constant $\kappa > 0$ such that, for sufficiently large $n$,
$$\inf\{ \bm{u}^\T\bm{\Sigma}_n\bm{u}: \Vert \bm{u} \Vert_2 =1, \bm{u} \in \mathcal{C}\} \ge 2\kappa.$$
\end{assumption}
It is not hard to see that this condition holds if the smallest eigenvalue of $\bm{\Sigma}_n$ is strictly bounded away from 0 for sufficiently large $n$. Furthermore, if vectorial functions $\bm{f}_i = \bm{f}$ do not vary with $i$ then $\{\bm{x}_i\}_{i=1}^n$ is a stationary time series as $\bm{x}_i = \bm{f}(\bm{Z}_i)$ are functions of the underlying Markov chain $\{\bm{Z}_i\}_{i=1}^n$. In this case, $\bm{\Sigma}_n = \bm{\Sigma} = \E[\bm{x}_1\bm{x}_1^\T]$, and Assumption \ref{asm4} holds if the smallest eigenvalue of $\bm{\Sigma}$ is strictly bounded away from 0.

Now we are ready to present the main result of this paper. Note that ``w.h.p." stands for ``with high probability $1-o(1)$".
\begin{theorem}\label{thm}
Suppose Assumptions \ref{asm1}-\ref{asm4} hold and
$$s \sqrt{\frac{1+\gamma}{1-\gamma} \cdot \frac{\log d}{n}} = o(1).$$
Then the AHR estimator $\widehat{\bm{\beta}}_{\tau, \lambda}$ in \eqref{Huber1} with robustification and regularization parameters
\begin{align*}
\tau \asymp \left(\frac{1-\gamma}{1+\gamma} \cdot \frac{n}{\log d}\right)^{1/(1+\min\{\delta,1\})}, ~~~\lambda \asymp \left(\frac{1+\gamma}{1-\gamma} \cdot \frac{\log d}{n}\right)^{\min\{\delta,1\}/(1+\min\{\delta,1\})}
\end{align*}
achieves estimation errors
$$\Vert \widehat{\bm{\beta}}_{\tau, \lambda} - \bm{\beta}^\star\Vert_1 \lesssim s\lambda,~~~\Vert \widehat{\bm{\beta}}_{\tau, \lambda} - \bm{\beta}^\star\Vert_2 \lesssim \sqrt{s} \lambda ~~~w.h.p..$$
\end{theorem}

Similar to the discovery of \citet*{sun2019adaptive} in the independent setup, there is a smooth phase transition of $\tau$-adaption and $\bm{\beta}^\star$-estimation. A small $0 < \delta < 1$ suffices for AHR to consistently estimate $\bm{\beta}^\star$, albeit the rates $s(\log d/n)^{\delta/(1+\delta)}$ or $s^{1/2}(\log d/n)^{\delta/(1+\delta)}$ (given fixed $\gamma < 1$) are slower than those for the case of $\delta = 1$. In the latter case of $\delta = 1$, AHR achieves rates $s(\log d/n)^{1/2}$ for $\ell_1$-error and $s^{1/2}(\log d/n)^{1/2}$ for $\ell_2$-error like classical high-dimensional regression methods; but AHR only requires bounded second moments of error $\varepsilon_i$, which is weaker than the sub-Gaussian error condition in classical high-dimensional regression methods.  A larger $\delta > 1$ gains no more estimation accuracy than $\delta = 1$.

The Markov dependence impacts on the adaption of the robustification parameter and the estimation of regression coefficients in the way that the sample size $n$ is discounted by a factor $(1-\gamma)/(1+\gamma) < 1$. In other words, to achieve comparable $\tau$-adaption and $\bm{\beta}^\star$-estimation, the required sample size increases by $(1+\gamma)/(1-\gamma)$ when moving from the independent setup to the Markov-dependent setup. Furthermore, this theorem allows $\gamma$ to approach to 1 as $n$ increases, so long as $s \sqrt{(1+\gamma)/(1-\gamma) \cdot \log d/n} \to 0$. Even though the spectral gap of the Markov chain is difficult to accurately compute in practice \citep*{hsu2015mixing}, Theorem \ref{thm} also apply if one replaces $\gamma$ with an inaccurate overestimate $\gamma' \ge \gamma$.

\section{Proof of Theorem \ref{thm}} \label{sec4}

We break the proof of Theorem \ref{thm} into three propositions. Proposition \ref{prop1} bounds the $\ell_1$- and $\ell_2$-errors of a generic $\ell_1$-regularized M-estimator $\widehat{\bm{\beta}}$ minimizing $\mathcal{L}(\bm{\beta}) + \lambda \Vert \bm{\beta} \Vert_1$. Roughly speaking, given a \textit{localized restricted eigenvalue} (LRE) condition, it establishes an $\ell_2$-error bound
$$\Vert \widehat{\bm{\beta}}_{\tau, \lambda} - \bm{\beta}^\star \Vert_2 \lesssim \sqrt{s}\Vert \nabla H_\tau(\bm{\beta}^\star)\Vert_\infty,$$
and a similar $\ell_1$-error bound with an additional factor $\sqrt{s}$, by applying Proposition 1 to $(\mathcal{L}, \widehat{\bm{\beta}})=(H_\tau, \widehat{\bm{\beta}}_{\tau, \lambda})$. This LRE condition requires strictly positive restricted eigenvalues over a local $\ell_1$-neighborhood. It is a simplified version of \citep*[Definition 2]{sun2019adaptive}. Another LRE condition in $\ell_2$-neighborhood has found applications in \citep*[Definition 4.1]{fan2018lamm}.

\begin{proposition} \label{prop1}
Consider a $\ell_1$-regularized minimizer $\widehat{\bm{\beta}} = \arg \min \mathcal{L}(\bm{\beta}) + \lambda \Vert \bm{\beta}\Vert_1$ of a convex, twice differentiable function $\mathcal{L}: \mathbb{R}^d \mapsto \mathbb{R}$. Suppose the following two conditions hold:
\begin{enumerate}[label=(\roman*)]
\item (localized restricted eigenvalue) Recall that $S$ is the support of $\bm{\beta}^\star$ and that $\mathcal{C}$ is defined as \eqref{cone}. There exists a constant $\kappa > 0$ such that
$$\inf\{ \bm{u}^\T \nabla^2 \mathcal{L}(\bm{\beta}) \bm{u}: ~\bm{u} \in \mathcal{C}, ~\Vert \bm{u}\Vert_2=1, ~\Vert \bm{\beta} - \bm{\beta}^\star \Vert_1 \le 48s\lambda/\kappa\} \ge \kappa.$$
\item A $s$-sparse $\bm{\beta}^\star$ satisfies $\Vert \nabla \mathcal{L}(\bm{\beta}^\star)\Vert_\infty \le \lambda/2$.
\end{enumerate}
Then
$$\Vert \widehat{\bm{\beta}} - \bm{\beta}^\star \Vert_1 \le 48s\lambda/\kappa,~~~\Vert \widehat{\bm{\beta}} - \bm{\beta}^\star\Vert_2 \le 12\sqrt{s}\lambda/\kappa.$$
\end{proposition}

To facilitate the proof of Theorem \ref{thm} as an application of Proposition \ref{prop1} to $(\mathcal{L}, \widehat{\bm{\beta}})=(H_\tau, \widehat{\bm{\beta}}_{\tau, \lambda})$, we establish the LRE condition for the Hessian matrix of the Huber loss $\nabla^2 H_\tau(\bm{\beta})$ in Proposition \ref{prop2}  and bound its gradient $\nabla H_\tau(\bm{\beta}^\star)$ in Proposition \ref{prop3}.

\begin{proposition} \label{prop2}
Under Assumptions \ref{asm1}-\ref{asm4}, we have
\begin{align*}
&~~~ \inf\{ \bm{u}^\T\nabla^2 H_\tau(\bm{\beta}) \bm{u}: \Vert \bm{u} \Vert_2 = 1, \bm{u} \in \mathcal{C}, \Vert \bm{\beta} - \bm{\beta}^\star \Vert_1 \le r\}\\
&= 2\kappa - sO_p\left(\sqrt{\frac{1+\gamma}{1-\gamma} \cdot \frac{\log d}{n}} + \frac{1}{\tau^{1+\delta}} + \frac{r^2}{\tau^2}\right).
\end{align*}
\end{proposition}

\begin{proposition} \label{prop3}
Under Assumptions \ref{asm1}-\ref{asm3}, for some constant $C > 0$
\begin{align*}
\Vert \nabla H_\tau(\bm{\beta}^\star) \Vert_\infty
&\le \sqrt{\frac{1+\gamma}{1-\gamma}\cdot \frac{2\sigma^2 v_{\min\{\delta,1\}} \tau^{1-\min\{\delta,1\}} \log d}{n}}\\
&~~~+ \frac{1+\gamma}{1-\gamma}\cdot \frac{20 \tau \log d}{n} + C\tau^{-\min\{\delta,1\}}, ~~~w.h.p..
\end{align*}
\end{proposition}

Since $\Vert \widehat{\bm{\beta}}_{\tau, \lambda} - \bm{\beta}^\star \Vert_2 \lesssim \sqrt{s}\Vert \nabla H_\tau(\bm{\beta}^\star)\Vert_\infty$, optimizing the bound of $\Vert \nabla H_\tau(\bm{\beta}^\star) \Vert_\infty$ in Proposition \ref{prop3} over $\tau$ get the optimal $\ell_2$-error bound for $\widehat{\bm{\beta}}_{\tau, \lambda}$. Collecting these pieces together proves Theorem \ref{thm}.

\begin{proof}[Proof of Theorem \ref{thm}]
Setting
\begin{align*}
\tau &= \left(\frac{1-\gamma}{1+\gamma} \cdot \frac{n}{\log d}\right)^{1/(1+\min\{\delta,1\})},\\
\lambda &= 2(\sqrt{2\sigma^2 v_{\min\{\delta,1\}}} + 20 + C)\tau^{-\min\{\delta,1\}},\\
r &= 48s\lambda/\kappa
\end{align*}
in Propositions \ref{prop2}-\ref{prop3} yields
\begin{equation*}
\begin{split}
& \inf\{ \bm{u}^\T\nabla^2 H_\tau(\bm{\beta}) \bm{u}: \Vert \bm{u} \Vert_2 = 1, \bm{u} \in \mathcal{C}, \Vert \bm{\beta} - \bm{\beta}^\star \Vert_1 \le r\} \ge \kappa,~~~w.h.p..\\
& \Vert \nabla H_\tau(\bm{\beta}^\star)\Vert_\infty \le \lambda / 2,~~~w.h.p..
\end{split}
\end{equation*}
Applying Proposition 1 to $(\mathcal{L}, \widehat{\bm{\beta}}) = (H_\tau,\widehat{\bm{\beta}}_{\tau, \lambda})$ completes the proof.
\end{proof}

\section{Other Technical Proofs} \label{sec5}

\subsection{Proof of Proposition \ref{prop1}}

The proof of Proposition \ref{prop1} consists of two steps.
\begin{enumerate}[label=(\alph*)]
\item $\widehat{\bm{\beta}} - \bm{\beta}^\star \in \mathcal{C}$, where $\mathcal{C}$ is defined in \eqref{cone}.
\item $\Vert \widehat{\bm{\beta}} - \bm{\beta}^\star \Vert_1 \le 4\sqrt{s}\Vert \widehat{\bm{\beta}} - \bm{\beta}^\star \Vert_2 \le r := 48s\lambda/\kappa$.
\end{enumerate}
Statement (a)  asserts that $\widehat{\bm{\beta}}$ belongs to the $\ell_1$ convex cone $\mathcal{C}$ around the estimand $\bm{\beta}^\star$. Its proof uses merely the convexity of $\mathcal{L}$, the optimality of $\widehat{\bm{\beta}}$, and H\"{o}lder's inequality for vectors, all of which are routine in the literature of high-dimensional linear regression. \citet*[Lemma 8]{sun2019adaptive} is a slightly stronger version of statement (a). Here we give a neat proof of statement (a) for the sake of readability.  Statement (b) provides similar but general results compared to \citet*[Theorem 8]{sun2019adaptive}: the former allows Markov-dependent covariate vectors with bounded forth moments (see Assumption \ref{asm2}) to fulfill the LRE condition in Proposition \ref{prop1}, while the latter requires a stronger condition that covariate vectors are i.i.d. and sub-Gaussian (see \citet*[Condition 5 and Theorem 8]{sun2019adaptive}). Moreover, our proof of statement (b) is simpler than that of \citet*[Theorem 8]{sun2019adaptive}, as the latter relies on an auxiliary result of the symmetric Bregman divergence (see \citet*[Lemma 2]{sun2019adaptive}).

\begin{proof}[Proof of Proposition \ref{prop1}(a)]
By the optimality of $\widehat{\bm{\beta}}$,
$$\mathcal{L}(\widehat{\bm{\beta}}) - \mathcal{L}(\bm{\beta}^\star) \le \lambda(\Vert \bm{\beta}^\star \Vert_1 - \Vert \widehat{\bm{\beta}} \Vert_1).$$
By the convexity of $\mathcal{L}$, H\"{o}lder's inequality, and the condition that $\Vert \nabla \mathcal{L}(\bm{\beta}^\star)\Vert_\infty \le \lambda/2$,
$$\mathcal{L}(\widehat{\bm{\beta}}) - \mathcal{L}(\bm{\beta}^\star) \ge \langle \nabla \mathcal{L}(\bm{\beta}^\star), \widehat{\bm{\beta}} -  \bm{\beta}^\star \rangle \ge - \Vert \nabla \mathcal{L}(\bm{\beta}^\star)\Vert_\infty \Vert \widehat{\bm{\beta}} -  \bm{\beta}^\star \Vert_1 \ge -\lambda \Vert \widehat{\bm{\beta}} -  \bm{\beta}^\star \Vert_1/2.$$
It follows that
$$2(\Vert \widehat{\bm{\beta}} \Vert_1 - \Vert \bm{\beta}^\star \Vert_1) \le \Vert \widehat{\bm{\beta}} - \bm{\beta}^\star \Vert_1.$$
On the left-hand side, using the fact that $\bm{\beta}^\star_{S^c} = \bm{0}$,
\begin{align*}
\Vert \widehat{\bm{\beta}} \Vert_1 - \Vert \bm{\beta}^\star \Vert_1
&= \Vert \widehat{\bm{\beta}}_{S^c} \Vert_1 + \Vert \widehat{\bm{\beta}}_S \Vert_1  - \Vert \bm{\beta}^\star_S \Vert_1\\
&= \Vert (\widehat{\bm{\beta}} - \bm{\beta}^\star)_{S^c} \Vert_1 + \Vert \bm{\beta}^\star_S + (\widehat{\bm{\beta}} - \bm{\beta}^\star)_S \Vert_1 - \Vert \bm{\beta}^\star_S \Vert_1\\
&\ge \Vert (\widehat{\bm{\beta}} - \bm{\beta}^\star)_{S^c} \Vert_1 - \Vert (\widehat{\bm{\beta}} - \bm{\beta}^\star)_S \Vert_1
\end{align*}
On the right-hand side,
$$\Vert \widehat{\bm{\beta}} - \bm{\beta}^\star \Vert_1 = \Vert (\widehat{\bm{\beta}} - \bm{\beta}^\star)_{S^c} \Vert_1 + \Vert (\widehat{\bm{\beta}} - \bm{\beta}^\star)_S \Vert_1.$$
Putting the last three displays together and rearranging terms completes the proof.
\end{proof}

\begin{proof}[Proof of Proposition \ref{prop1}(b)]
First, it follows from step (a) that $\Vert \widehat{\bm{\beta}} -  \bm{\beta}^\star \Vert_1 \le 4\sqrt{s} \Vert \widehat{\bm{\beta}} -  \bm{\beta}^\star \Vert_2$. It is left to show $4\sqrt{s} \Vert \widehat{\bm{\beta}} -  \bm{\beta}^\star \Vert_2 \le r$. Suppose for the sake of contradiction that $4\sqrt{s}\Vert \widehat{\bm{\beta}} - \bm{\beta}^\star \Vert_2 > r$. By the optimality of $\widehat{\bm{\beta}}$ and the integral form of the Taylor expansion,
\begin{align*}
0 &\ge [\mathcal{L}(\widehat{\bm{\beta}}) + \lambda \Vert \widehat{\bm{\beta}} \Vert_1] - [\mathcal{L}(\bm{\beta}^\star) + \lambda \Vert \bm{\beta}^\star \Vert_1]\\
&= \lambda(\Vert \widehat{\bm{\beta}} \Vert_1 - \Vert \bm{\beta}^\star \Vert_1) + \langle \nabla \mathcal{L}(\bm{\beta}^\star), \widehat{\bm{\beta}} -  \bm{\beta}^\star \rangle\\
&~~~+ \int_0^1 (1-t) (\widehat{\bm{\beta}} -  \bm{\beta}^\star)^\T \nabla^2 \mathcal{L}(\bm{\beta}^\star + t(\widehat{\bm{\beta}}-\bm{\beta}^\star))(\widehat{\bm{\beta}} -  \bm{\beta}^\star)dt
\end{align*}
For the first term,
$$\lambda(\Vert \widehat{\bm{\beta}} \Vert_1 - \Vert \bm{\beta}^\star \Vert_1) \ge -\lambda \Vert \widehat{\bm{\beta}} - \bm{\beta}^\star \Vert_1 \ge -4\sqrt{s}\lambda \Vert \widehat{\bm{\beta}} - \bm{\beta}^\star \Vert_2.$$
For the second term, from the condition that $\Vert \nabla \mathcal{L}(\bm{\beta}^\star) \Vert_\infty \le \lambda/2$, it follows that
\begin{align*}
|\langle \nabla \mathcal{L}(\bm{\beta}^\star), \widehat{\bm{\beta}} -  \bm{\beta}^\star \rangle|
& \le \Vert (\nabla \mathcal{L}(\bm{\beta}^\star))_S \Vert_2 \Vert (\widehat{\bm{\beta}} -  \bm{\beta}^\star)_S \Vert_2 + \Vert (\nabla \mathcal{L}(\bm{\beta}^\star))_{S^c} \Vert_\infty \Vert (\widehat{\bm{\beta}} -  \bm{\beta}^\star)_{S^c} \Vert_1\\
&\le \sqrt{s} \Vert (\nabla \mathcal{L}(\bm{\beta}^\star))_S \Vert_\infty \Vert (\widehat{\bm{\beta}} -  \bm{\beta}^\star)_S \Vert_2 + \Vert (\nabla \mathcal{L}(\bm{\beta}^\star))_{S^c} \Vert_\infty \Vert (\widehat{\bm{\beta}} -  \bm{\beta}^\star)_{S^c} \Vert_1\\
&\le \frac{\sqrt{s} \lambda}{2} \Vert (\widehat{\bm{\beta}} -  \bm{\beta}^\star)_S \Vert_2 + \frac{3\lambda}{2} \Vert (\widehat{\bm{\beta}} -  \bm{\beta}^\star)_S \Vert_1 \le 2\sqrt{s}\lambda \Vert \widehat{\bm{\beta}} -  \bm{\beta}^\star \Vert_2.
\end{align*}
Proceed to lower bound the third term. To this end, note that $r / 4\sqrt{s} \Vert \widehat{\bm{\beta}}  - \bm{\beta}^\star \Vert_2 < 1$ by the initial assumption. For any $0 \le t \le r / 4\sqrt{s} \Vert \widehat{\bm{\beta}}  - \bm{\beta}^\star \Vert_2$,
$$\Vert [\bm{\beta}^\star + t(\widehat{\bm{\beta}}-\bm{\beta}^\star)] - \bm{\beta}^\star \Vert_1 \le t \Vert \widehat{\bm{\beta}}-\bm{\beta}^\star \Vert_1 \le t \times 4\sqrt{s}\Vert \widehat{\bm{\beta}}-\bm{\beta}^\star \Vert_2 \le r.$$
Combining it with step (a) that $\widehat{\bm{\beta}} - \bm{\beta}^\star \in \mathcal{C}$ and the LRE condition yield a lower bound for the third term
\begin{align*}
&~~~\int_0^1 (1-t) (\widehat{\bm{\beta}} -  \bm{\beta}^\star)^\T \nabla^2 \mathcal{L}(\bm{\beta}^\star + t(\widehat{\bm{\beta}}-\bm{\beta}^\star))(\widehat{\bm{\beta}} -  \bm{\beta}^\star)dt\\
&\ge \int_0^{r / 4\sqrt{s} \Vert \widehat{\bm{\beta}}  - \bm{\beta}^\star \Vert_2} (1-t) (\widehat{\bm{\beta}} -  \bm{\beta}^\star)^\T \nabla^2 \mathcal{L}(\bm{\beta}^\star + t(\widehat{\bm{\beta}}-\bm{\beta}^\star))(\widehat{\bm{\beta}} -  \bm{\beta}^\star)dt\\
&\ge \int_0^{r / 4\sqrt{s} \Vert \widehat{\bm{\beta}}  - \bm{\beta}^\star \Vert_2} (1-t) \kappa \Vert \widehat{\bm{\beta}}-\bm{\beta}^\star\Vert_2^2 dt = \frac{\kappa r}{4\sqrt{s}} \Vert \widehat{\bm{\beta}}-\bm{\beta}^\star\Vert_2 - \frac{\kappa r^2}{32s}
\end{align*}
Putting the lower bounds of three terms together with the scaling of $r = 48s\lambda / \kappa$ yields
$$0 \ge \left(\frac{\kappa r}{4\sqrt{s}} - 6\sqrt{s}\lambda\right) \Vert \widehat{\bm{\beta}} -  \bm{\beta}^\star \Vert_2 - \frac{\kappa r^2}{32s} = 6\sqrt{s}\lambda\Vert \widehat{\bm{\beta}} -  \bm{\beta}^\star \Vert_2 - \frac{3\lambda r}{2},$$
which contradicts to the initial assumption that $4\sqrt{s}\Vert \widehat{\bm{\beta}} - \bm{\beta}^\star \Vert_2 > r$.
\end{proof}

\subsection{Proof of Proposition \ref{prop2}}
We first present two lemmas, which are useful in the proof of Proposition \ref{prop2}. Proofs of these two lemmas are put at the end of this subsection.

\begin{lemma} \label{lem1}
Under Assumptions \ref{asm1}-\ref{asm2},
$$\max_{1 \le j, k \le d} \left|\frac{1}{n}\sum_{i=1}^n x_{ij}x_{ik} - \frac{1}{n}\sum_{i=1}^n \E[x_{ij}x_{jk}]\right| = O_p\left(\sqrt{\frac{1+\gamma}{1-\gamma} \cdot \frac{\log d}{n}}\right).$$
\end{lemma}

\begin{lemma} \label{lem2}
Under Assumptions \ref{asm1}-\ref{asm3},
$$\frac{1}{n} \sum_{i=1}^n M^2(Z_i)1\{|\varepsilon_i| > \tau/2\}  \le \sigma^2 \left(\frac{2}{\tau}\right)^{1+\delta} v_\delta + O_p\left(\sqrt{\frac{1+\gamma}{1-\gamma} \cdot \frac{\log d}{n}}\right).$$
\end{lemma}

\begin{proof}[Proof of Proposition \ref{prop2}]
For any $\bm{\beta}$ such that $\Vert \bm{\beta} - \bm{\beta}^\star \Vert_1 \le r$,
\begin{align*}
1\{|y_i - \bm{x}_i^\T \bm{\beta}| > \tau \}
&\le 1\{|\varepsilon_i| > \tau/2 \} + 1\{|\bm{x}_i^\T (\bm{\beta} - \bm{\beta}^\star)| > \tau/2 \}\\
&\le 1\{|\varepsilon_i| > \tau/2 \} + 1\{ \Vert \bm{x}_i \Vert_\infty > \tau/2r \}\\
&\le 1\{|\varepsilon_i| > \tau/2 \} + 1\{ M(Z_i)  > \tau/2r \},
\end{align*}
where $M$ is the envelop function introduced by Assumption \ref{asm2}. Let
\begin{align*}
\epsilon_{1n} &= \max_{1 \le j, k \le d} \left| \left[\bm{\Sigma}_n - \frac{1}{n}\sum_{i=1}^n \bm{x}_i \bm{x}_i^\T\right]_{j, k}\right| = \max_{1 \le j, k \le d} \left| \frac{1}{n}\sum_{i=1}^n x_{ij}x_{ik} - \frac{1}{n}\sum_{i=1}^n \E[x_{ij}x_{ik}]\right|\\
\epsilon_{2n} &= \max_{1 \le j, k \le d} \left| \left[ \frac{1}{n}\sum_{i=1}^n \bm{x}_i \bm{x}_i^\T 1\{|\varepsilon_i| \ge \tau/2 \}\right]_{j, k}\right| \le \frac{1}{n} \sum_{i=1}^n M^2(Z_i)1\{|\varepsilon_i| \ge \tau/2 \}\\
\epsilon_{3n} &= \max_{1 \le j, k \le d} \left| \left[ \frac{1}{n}\sum_{i=1}^n \bm{x}_i \bm{x}_i^\T 1\{M(Z_i) \ge \tau/2r \}\right]_{j, k}\right| \le (2r/\tau)^2 \times \frac{1}{n} \sum_{i=1}^n M^4(Z_i).
\end{align*}
For any $\bm{u} \in \mathcal{C}$ such that $\Vert \bm{u} \Vert_2 = 1$, we have $\Vert \bm{u} \Vert_1 \le 4\sqrt{s}$. Thus,
\begin{align*}
\bm{u}^\T \nabla^2 H_\tau(\bm{\beta}) \bm{u}
&= \bm{u}^\T \left[ \frac{1}{n}\sum_{i=1}^n \bm{x}_i \bm{x}_i^\T 1\{|y_i - \bm{x}_i^\T \bm{\beta}| \le \tau \}\right] \bm{u},\\
&\ge \bm{u}^\T \bm{\Sigma}_n \bm{u} - \bm{u}^\T\left[\bm{\Sigma}_n - \frac{1}{n}\sum_{i=1}^n \bm{x}_i \bm{x}_i^\T\right]\bm{u}\\
&~~~~~~- \bm{u}^\T \left[\frac{1}{n}\sum_{i=1}^n \bm{x}_i\bm{x}_i^\T 1\{|\varepsilon_i| > \tau/2\} \right]\bm{u}\\
&~~~~~~- \bm{u}^\T \left[\frac{1}{n}\sum_{i=1}^n \bm{x}_i\bm{x}_i^\T 1\{M(Z_i) > \tau/2r\} \right]\bm{u}.\\
&\ge \bm{u}^\T \bm{\Sigma}_n \bm{u} - \epsilon_{1n}\Vert \bm{u} \Vert_1^2 - \epsilon_{2n}\Vert \bm{u} \Vert_1^2 - \epsilon_{3n}\Vert \bm{u} \Vert_1^2,\\
&\ge 2\kappa - 16s(\epsilon_{1n} + \epsilon_{2n} + \epsilon_{3n}).
\end{align*}
Further bounding $\epsilon_{1n}$ by Lemma \ref{lem1}, $\epsilon_{2n}$ by Lemma \ref{lem2} and
$$\epsilon_{3n} \le  (2r/\tau)^2 \times (\sigma^4 + o_p(1))$$
by the law of large number for Markov chains \citep[Theorem 17.1.2]{meyn2012markov} completes the proof.
\end{proof}

\begin{proof}[Proof of Lemma \ref{lem1}]
Define a truncation operator
\begin{equation} \label{truncation}
\mathcal{T}_t(w) =
\begin{cases}
-t & \mbox{if } w < -t\\
w & \mbox{if } |w| \le t\\
+t & \mbox{if } w > +t.\\
\end{cases}
\end{equation}
For each $1 \le j \le d$ and each $1 \le k \le d$,
$$\left|\frac{1}{n}\sum_{i=1}^n x_{ij}x_{ik} - \frac{1}{n}\sum_{i=1}^n \E[x_{ij}x_{jk}] \right|\le D_{1jk} + D_{2jk} + D_{3jk},$$
where
\begin{align*}
D_{1jk} &=  \left|\frac{1}{n}\sum_{i=1}^n \E \mathcal{T}_t[x_{ij}x_{ik}] - \frac{1}{n}\sum_{i=1}^n \E[x_{ij}x_{ik}]\right|,\\
D_{2jk} &= \left|\frac{1}{n}\sum_{i=1}^n \mathcal{T}_t[x_{ij}x_{ik}] - \frac{1}{n}\sum_{i=1}^n x_{ij}x_{ik} \right|,\\
D_{3jk} &=  \left|\frac{1}{n}\sum_{i=1}^n \mathcal{T}_t[x_{ij}x_{ik}] - \frac{1}{n}\sum_{i=1}^n \E \mathcal{T}_t[x_{ij}x_{ik}] \right|.
\end{align*}
Using the fact that $|\mathcal{T}_t(w) - w| \le |w|1\{|w|>t\} \le |w|^2/t$, Cauchy-Schwarz inequality and the bounds on fourth moment of the envelop function $M:z \mapsto \mathbb{R}$ in Assumption \ref{asm2},
\begin{align*}
\max_{j,k} D_{1jk} &\le \max_{j,k}  \frac{1}{tn}\sum_{i=1}^n \E[|x_{ij}x_{ik}|^2] \le \max_{j,k} \frac{1}{t} \sqrt{\frac{1}{n}\sum_{i=1}^n \E[x_{ij}^4] \times \frac{1}{n}\sum_{i=1}^n \E[x_{ik}^4]}\\
&\le \frac{1}{tn} \sum_{i=1}^n \E[M^4(Z_i)] \le \frac{\sigma^4}{t}.
\end{align*}
By a similar argument and the law of large number for Markov chains \citep[Theorem 17.1.2]{meyn2012markov}
$$\max_{j,k} D_{2jk} \le \frac{1}{tn} \sum_{i=1}^n M^4(Z_i) = \frac{\sigma^4 + o_p(1)}{t}.$$
Noting that $|\mathcal{T}_t[W] - \E \mathcal{T}_t[W]| \le 2t$ almost surely, we apply the Bernstein's inequality for Markov chains in \citep*[Theorem 1.1]{jiang2018bernstein} and yields
$$\P\left(D_{3jk} > \epsilon \right) \le 2\exp{\left(-\frac{n\epsilon^2}{\frac{1+\gamma}{1-\gamma} \cdot V_{n,t} + 10t\epsilon }\right)} \le 2\exp{\left(-\frac{n\epsilon^2}{\frac{1+\gamma}{1-\gamma} \cdot \sigma^4 + 10t\epsilon }\right)},$$
where
$$V_{n,t} = \frac{1}{n} \sum_{i=1}^n \text{Var}\{\mathcal{T}_t[x_{ij}x_{ik}]\} \le \frac{1}{n} \sum_{i=1}^n \E[x_{ij}^2x_{ik}^2] \le \sigma^4.$$
A union bound delivers that
$$\P\left( \max_{j,k} D_{3jk} > \epsilon \right) \le 2p^2\exp{\left(-\frac{n\epsilon^2}{\frac{1+\gamma}{1-\gamma} \cdot \sigma^4 + 10t\epsilon }\right)}.$$
Let $t = \frac{1}{10c}\sqrt{\frac{1-\gamma}{1+\gamma}\cdot \frac{n}{\log d}}$, and $\epsilon = c\sqrt{\frac{1+\gamma}{1-\gamma}\cdot \frac{\log d}{n}}$ for some $c > \sqrt{2(\sigma^4 + 1)}$. Then
$$\max_{j,k} D_{3jk} \ge c\sqrt{\frac{1+\gamma}{1-\gamma}\cdot \frac{\log d}{n}}$$
with probability at most
$$2\exp{\left(-\left(\frac{c^2}{\sigma^4 + \frac{1-\gamma}{1+\gamma}}-2\right)\log d \right)} \le 2\exp{\left(-\left(\frac{c^2}{\sigma^4 + 1} - 2\right)\log d\right)}.$$
Collecting these pieces together completes the proof.
\end{proof}

\begin{proof}[Proof of Lemma \ref{lem2}]
Note that $\{(Z_i, \varepsilon_i)\}_{i=1}^n$ are identically distributed. Write
\begin{align*}
\E\left[M^2(Z_1) 1\{|\varepsilon_1| > \tau/2\}\right]
&= \E\left[M^2(Z_1) \P(|\varepsilon_1| > \tau/2| Z_1)\right] \\
&\le \E\left[M^2(Z_1) \left(\frac{2}{\tau}\right)^{1+\delta} \E[|\varepsilon_1|^{1+\delta}|Z_1]\right] \\
&\le \E[M^2(Z_1)] \left(\frac{2}{\tau}\right)^{1+\delta} v_\delta \le \sigma^2 \left(\frac{2}{\tau}\right)^{1+\delta} v_\delta.
\end{align*}
It is left to show
$$\left|\frac{1}{n}\sum_{i=1}^n M^2(Z_i)1\{|\varepsilon_i| > \tau/2\} - \E\left[M^2(Z_1)1\{|\varepsilon_1| > \tau/2\}\right]\right| = O_p\left(\sqrt{\frac{1+\gamma}{1-\gamma} \cdot \frac{\log d}{n}}\right).$$
Recall that $\mathcal{T}_t$ is the truncation operator with threshold $t$ defined in \eqref{truncation}. Break down the quantity on the left-hand side of the last display into three terms.
$$D_1 + D_2 + D_3,$$
where
\begin{align*}
D_1 &=  \left|\E \mathcal{T}_t[M^2(Z_1) 1\{|\varepsilon_1| > \tau/2\}] - \E[M^2(Z_1)1\{|\varepsilon_1| > \tau/2\}]\right|,\\
D_2 &= \left|\frac{1}{n}\sum_{i=1}^n \mathcal{T}_t[M^2(Z_i)1\{|\varepsilon_i| > \tau/2\}] - \frac{1}{n}\sum_{i=1}^n M^2(Z_i) 1\{|\varepsilon_i| > \tau/2\} \right|,\\
D_3 &=  \left|\frac{1}{n}\sum_{i=1}^n \mathcal{T}_t[M^2(Z_i)1\{|\varepsilon_i| > \tau/2\}] - \E \mathcal{T}_t[M^2(Z_1)1\{|\varepsilon_1| > \tau/2\}] \right|.
\end{align*}
Using the fact that $|\mathcal{T}_t(w) - w| \le |w|1\{|w|>t\} \le |w|^2/t$, Cauchy-Schwarz inequality and the bounds on fourth moment of $M:z \mapsto \mathbb{R}$ in Assumption \ref{asm2},
$$D_1 \le \frac{\E[M^4(Z_1)]}{t} = \frac{\sigma^4}{t}.$$
By a similar argument and the law of large number for Markov chains \citep[Theorem 17.1.2]{meyn2012markov}
$$D_2 \le \frac{1}{tn} \sum_{i=1}^n M^4(Z_i) = \frac{\sigma^4 + o_p(1)}{t}.$$
Noting that $|\mathcal{T}_t[W] - \E \mathcal{T}_t[W]| \le 2t$ almost surely, we apply the Bernstein's inequality for geometrically ergodic Markov chains in \citep*[Theorem 1.1]{jiang2018bernstein} and yields
$$\P\left(D_3 > \epsilon \right) \le 2\exp{\left(-\frac{n\epsilon^2}{\frac{1+\gamma}{1-\gamma} \cdot V_t + 10t\epsilon }\right)} \le 2\exp{\left(-\frac{n\epsilon^2}{\frac{1+\gamma}{1-\gamma} \cdot \sigma^4 + 10t\epsilon }\right)},$$
where
$$V_t = \text{Var}\{\mathcal{T}_t[M^2(Z_1)\{|\varepsilon_2| > \tau/2\}]\} \le \E[M^4(Z_1)] = \sigma^4.$$
Let $t = \frac{1}{10c}\sqrt{\frac{1-\gamma}{1+\gamma}\cdot \frac{n}{\log d}}$, and $\epsilon = c\sqrt{\frac{1+\gamma}{1-\gamma}\cdot \frac{\log d}{n}}$ for some $c > 0$ then
$$D_3 \ge c\sqrt{\frac{1+\gamma}{1-\gamma}\cdot \frac{\log d}{n}}$$
with probability at most
$$\exp{\left(-\frac{c^2 \log d}{\sigma^4 + \frac{1-\gamma}{1+\gamma}}\right)}.$$
Collecting these pieces together completes the proof.
\end{proof}

\subsection{Proof of Proposition \ref{prop3}}
\begin{proof}[Proof of Proposition \ref{prop3}]
We merely consider the case of $0 < \delta \le 1$. The proof for the case of $\delta > 1$ is the same with that for $\delta = 1$. Recall that $\mathcal{T}_t$ is the truncation operator with threshold $t$ defined in \eqref{truncation}, and $h_\tau$ is the Huber loss with the robustification parameter $\tau$ defined in \eqref{Huber2}. For each $1 \le j \le d$,
$$-\nabla_j H_\tau(\bm{\beta}^\star) = -\frac{\partial H_\tau(\bm{\beta}^\star)}{\partial \beta_j} = \frac{1}{n} \sum_{i=1}^n \mathcal{T}_\tau(\varepsilon_i)x_{ij},$$
Note that $\E[\epsilon_i|Z_i] = 0$ almost surely in Assumption \ref{asm3}. Write
\begin{align*}
|\E[\nabla_j H_\tau(\bm{\beta}^\star)]|
&\le \frac{1}{n} \sum_{i=1}^n \E[\mathcal{T}_\tau(\varepsilon_i)|x_{ij}|] = \frac{1}{n} \sum_{i=1}^n \E[\E[\mathcal{T}_\tau(\varepsilon_i)|\bm{x}_i]|x_{ij}|]\\
&= \frac{1}{n} \sum_{i=1}^n \E[\{\E[\varepsilon_i|\bm{x}_i] - \E[\epsilon_i|Z_i]\} |x_{ij}|]\\
&\le \frac{1}{n} \sum_{i=1}^n \E[\E[|\varepsilon_i|1\{|\varepsilon_i|>\tau\}|Z_i]|x_{ij}|]\\
&\le \frac{1}{n} \sum_{i=1}^n \E[\E[|\varepsilon_i|^{1+\delta}\tau^{-\delta}|Z_i]|x_{ij}|] \le \sigma v_\delta \tau^{-\delta}
\end{align*}
Next, we use a similar argument to that in Lemmas \ref{lem1}-\ref{lem2} to bound the deviations of $\nabla_j H_\tau(\bm{\beta}^\star)$ from their expectations. Break down the deviation into three terms as follows.
$$|\nabla_j H_\tau(\bm{\beta}^\star) - \E[\nabla_j H_\tau(\bm{\beta}^\star)]| \le D_{1j} + D_{2j} + D_{3j},$$
where
\begin{align*}
D_{1j} &=  \left|\frac{1}{n}\sum_{i=1}^n \E \mathcal{T}_t[\mathcal{T}_\tau(\varepsilon_i)x_{ij}] - \frac{1}{n}\sum_{i=1}^n \E[\mathcal{T}_\tau(\varepsilon_i)x_{ij}]\right|,\\
D_{2j} &= \left|\frac{1}{n}\sum_{i=1}^n \mathcal{T}_t[\mathcal{T}_\tau(\varepsilon_i)x_{ij}] - \frac{1}{n}\sum_{i=1}^n \mathcal{T}_\tau(\varepsilon_i)x_{ij} \right|,\\
D_{3j} &=  \left|\frac{1}{n}\sum_{i=1}^n \mathcal{T}_t[\mathcal{T}_\tau(\varepsilon_i)x_{ij}] - \frac{1}{n}\sum_{i=1}^n \E \mathcal{T}_t[\mathcal{T}_\tau(\varepsilon_i)x_{ij}] \right|.
\end{align*}
Using the fact that $|\mathcal{T}_t(w) - w| \le |w|1\{|w|>t\} \le |w|^2/t$, Cauchy-Schwarz inequality, the bounds on moments in Assumptions \ref{asm2}-\ref{asm3},
$$\max_j D_{1j} \le \frac{1}{tn}\sum_{i=1}^n \E[|\mathcal{T}_\tau(\varepsilon_i)|^2M^2(Z_i)] \le  \frac{\E[\tau^{1-\delta}|\varepsilon_1|^{1+\delta}M^2(Z_1)]}{t}  \le \frac{\sigma^2 v_\delta \tau^{1-\delta}}{t}.$$
Note that the augmented sequence $\{(Z_i, \varepsilon_i)\}_{i=1}^n$ is a stationary Markov chain. By a similar argument to that in Lemmas \ref{lem1}-\ref{lem2} and the law of large number for Markov chains \citep[Theorem 17.1.2]{meyn2012markov}
\begin{align*}
\max_j D_{2j} &\le \frac{1}{tn}\sum_{i=1}^n \tau^{1-\delta} |\varepsilon_i|^{1+\delta}M^2(Z_i)\\
&= \frac{\tau^{1-\delta}(\E[|\varepsilon_1|^{1+\delta}M^2(Z_1)]+o_p(1))}{t} \le \frac{\tau^{1-\delta}(\sigma^2 v_\delta +o_p(1))}{t} .
\end{align*}
Noting that $|\mathcal{T}_t[W] - \E \mathcal{T}_t[W]| \le 2t$ almost surely, we apply the Bernstein's inequality for Markov chains in \citep*[Theorem 1.1]{jiang2018bernstein} and yields
$$\P\left(D_{3j} > \epsilon \right) \le 2\exp{\left(-\frac{n\epsilon^2}{\frac{1+\gamma}{1-\gamma} \cdot V_{n,t} + 10t\epsilon }\right)} \le 2\exp{\left(-\frac{n\epsilon^2}{\frac{1+\gamma}{1-\gamma} \cdot \sigma^2 v_\delta \tau^{1-\delta} + 10t\epsilon }\right)},$$
where
$$V_{n,t} = \frac{1}{n} \sum_{i=1}^n \text{Var}\{\mathcal{T}_t[\mathcal{T}_\tau(\varepsilon_i)x_{ij}]\} \le \frac{1}{n} \sum_{i=1}^n \E[|\mathcal{T}_\tau(\varepsilon_i)|^2x_{ij}^2] \le \sigma^2 v_\delta \tau^{1-\delta}.$$
A union bound delivers that
$$\P\left( \max_{j} D_{3j} > \epsilon \right) \le 2d\exp{\left(-\frac{n\epsilon^2}{\frac{1+\gamma}{1-\gamma} \cdot \sigma^2 v_\delta \tau^{1-\delta} + 10t\epsilon }\right)}.$$
Let $t = \frac{1+\gamma}{1-\gamma} \cdot \tau$ then
$$\max_{j} D_{3j} > \sqrt{\frac{1+\gamma}{1-\gamma}\cdot \frac{2\sigma^2 v_\delta \tau^{1-\delta} \log d}{n}} + \frac{1+\gamma} {1-\gamma}\cdot \frac{20 \tau \log d}{n}$$
with probability at most $2/d$. Collecting these pieces together completes the proof.
\end{proof}

\section{Discussion} \label{sec6}

Heavy-tailed errors and dependent observations are two stylized features of many real high-dimensional data. However, the current framework of high-dimensional regression assumes sub-Gaussian tails of errors and independent observations for the convenience of theoretical analyses. Our long-term goal is to generalize the current framework to cover real cases of both heavy-tailed errors and dependent observations. While \citep*{sun2019adaptive} proposed AHR for the robust estimation against heavy-tailed errors, this paper makes progresses in the vertical direction and deals with dependent observations.

The technical key of AHR is the Bernstein-type inequality to quantitatively characterize the tradeoff between bias and robustness in the Huber regression so that the optimal robustification parameter is adaptively chosen. For the independent setup, Bernstein-type inequalities, especially those for unbounded random variables, have been well established. However, for the Markov-dependent setup, the existing Bernstein-type inequalities in the literature work for bounded functions of Markov chains only \citep*{paulin2015concentration,jiang2018bernstein}. We develop a truncation argument to extend the Bernstein-type inequality in \citep*{jiang2018bernstein} to unbounded functions of Markov chains. This truncation argument is potentially useful for other high-dimensional statistical problems in the Markov-dependent setup, e.g., the high-dimensional covariance matrix estimation on Markov-dependent samples. 

In the Markov-dependent setup, we find that the $\tau$-adaption and the $\bm{\beta}$-estimation of AHR exhibits a similar phase transition to that in the independent setup. The only difference is the sample size $n$ is discounted by a factor $(1-\gamma)/(1+\gamma)$, where $\gamma$ is the norm of Markov operator acting on the Hilbert space $L_2(\pi)$ and measures the Markov dependence. If $\gamma = 0$ then the conclusion recovers the main result of \citet*{sun2019adaptive} for the independent data samples. If $\gamma \to 1$ as $n \to \infty$ then the conclusion in Theorem \ref{thm} is still valid so long as
$$s\sqrt{\frac{1+\gamma}{1-\gamma} \cdot \frac{\log d}{n}} \to 0.$$
It is also valid if $\gamma$ is replaced with an overestimate $\gamma' \ge \gamma$. Inspired by the discussion on the optimal variance proxies of concentration inequalities for functions of Markov chains in \citep*{jiang2018bernstein}, we conjecture that the discounting factor $(1-\gamma)/(1+\gamma)$ is optimally sharp when translating the theories of high-dimensional linear regression from the independent setup to the Markov-dependent setup.

The theoretical and practical applicabilities of AHR on time series data with more complicated dependence structure are the potential directions of future works. Of particular interest is how the dependence of the time series data get involved in the $\tau$-adaption and the $\bm{\beta}$-estimation of AHR. However, the theoretical analysis of AHR in more complicated setups still hinges on the availability of more general Bernstein-type inequalities. Unfortunately, there are few concentration inequalities fulfilling the goal to theoretically justify AHR and other high-dimensional regression methods on data with a general dependence structure.



\bibliographystyle{ims}
\bibliography{ref}


\end{document}